\documentclass[11pt]{amsart}
\usepackage{amsfonts}
\usepackage{}
\usepackage{mathrsfs}
\usepackage{bbm}
\usepackage{amssymb}
\usepackage{indentfirst}
\usepackage{latexsym,amsfonts,amssymb,amsmath}
\newtheorem{theorem}{Theorem}[section]
\newtheorem{lemma}[theorem]{Lemma}

\theoremstyle{definition}
\newtheorem{definition}[theorem]{Definition}
\newtheorem{proposition}[theorem]{Proposition}
\theoremstyle{remark}
\newtheorem{remark}[theorem]{Remark}

\begin{document}

\title
{The construction of Hartman-Mycielski in topological gyrogroups}
\author{Ying-Ying Jin}\thanks{}
\address{(Y.Y. Jin) School of Mathematics and Computational Science, Wuyi University, Jiangmen 529020, P.R. China} \email{yingyjin@163.com}

\author{Li-Hong Xie*}\thanks{* The corresponding author.}
\address{(L.H. Xie) School of Mathematics and Computational Science, Wuyi University, Jiangmen 529020, P.R. China} \email{xielihong2011@aliyun.com}


\thanks{
This work is supported by Natural Science Foundation of China (Grant Nos. 11861018, 11871379, 11526158), the Natural Science Foundation of Guangdong
Province under Grant (No.2018A030313063), The Innovation Project of Department of Education of Guangdong Province  (No:2018KTSCX231), and key project of Natural Science Foundation of Guangdong Province Universities (2019KZDXM025).}



\begin{abstract}
 The concept of gyrogroups is a generalization of groups which
do not explicitly have associativity. Recently, Wattanapan et al consider the construction of Hartman-Mycielski in strongly topological gyrogroups. In this paper, we extend their results in topological gyrogroups. We mainly, among other results, prove that every Hausdorff topological gyrogroup $G$ can be
embedded as a closed subgyrogroup of a Hausdorff path-connected and locally path-connected topological
gyrogroup $G^\bullet$.
\end{abstract}

\maketitle

\section{Introduction}
The concept of gyrogroup was discovered by Ungar when he
studied the Einstein velocity addition \cite{Ung}.
A gyrogroup is a generalization of a group in the sense that it is a groupoid with an identity and inverses,
but the associative law is redefined by more general definitions which are the left gyroassociative law and
the left loop property.
The gyrogroup does not form a group since it is neither associative nor commutative.
Nevertheless, Ungar showed that gyrogroups are rich in algebraic structure and
encodes a group-like structure, namely the gyrogroup structures \cite{Ung}.
Many important characteristics of gyrogroups have been intensively studied in \cite{Fer1}, \cite{Fer2}, \cite{Suk1}, \cite{Suk2} and \cite{Suk3}.
Suksumran and Wiboonton have studied some basic algebraic properties of gyrogroups, for example, the isomorphism theorems, Cayley's Theorem, Lagrange's Theorem, the
gyrogroup actions, etc. in \cite{Suk1}, \cite{Suk2} and \cite{Suk3}.
Most of these properties are similar to those in classical group theory.
Atiponrat \cite{Atip} extended the idea of topological groups to topological gyrogroups as gyrogroups with a topology such that its binary operation is jointly continuous and the operation of taking the inverse is continuous.
Some basic properties of topological gyrogroups are studied in some detail; see, for instance \cite{Atip, Atip1, Cai}.

In 1958, Hartman and Mycielski proved that every Hausdorff topological group can be embedded into a Hausdorff path-connected, locally path-connected groups \cite{HM}. Recently, Wattanapan  et al extend this result to strongly topological gyrogroups \cite{Wat}.

{\bf Construction of Hartman-Mycielski \cite{Wat}:} Let $G$ be a gyrogroup with identity $e$ and let $J = [0, 1)$. A function $f : J\rightarrow G$ is a step function
if there are real numbers $a_0, a_1,..., a_n$ such that $0 = a_0 < a_1 < \cdot\cdot\cdot< a_n = 1$ and $f$ is constant on
$[a_k, a_{k+1})$ for all $k = 0, 1,..., n-1$.
 Henceforward, when we say that $A =\{a_0, a_1,..., a_n\}$ is a partition
of $J$, we include the condition that $0 = a_0 < a_1 < \cdot\cdot\cdot< a_n = 1$. Denote by $G^\bullet$ the set of all step
functions. Define an operation $\oplus$ on $G^\bullet$ by
$$( f \oplus g)(r) = f (r) \oplus g(r), r\in J\quad\quad\quad(1)$$
for all $f , g \in G^\bullet$. Let $f , g \in G^\bullet$. It is easy to see that $f \oplus g$ is again a step function.

\begin{theorem}\cite{Wat}\label{Th1}
$G^\bullet$ forms a gyrogroup in Construction of Hartman-Mycielski. If $G$ is a Hausdorff strongly topological gyrogroup, then $G^\bullet$ can become a Hausdorff path-connected, locally path-connected strongly topological gyrogroup containing $G$ as a closed subgyrogroup.
\end{theorem}

In this paper, we mainly consider Theorem \ref{Th1} in topological gyrogroups and prove that: If $G$ is a Hausdorff topological gyrogroup, then $G^\bullet$ can become a Hausdorff path-connected, locally path-connected topological gyrogroup containing $G$ as a closed subgyrogroup.

All spaces are not assumed to satisfy any separation axiom unless otherwise stated.

\section{Some basic facts and definitions}

In this section, some basic definitions and results are stated. Let $G$ be a nonempty set, and let $\oplus  : G  \times G \rightarrow G $ be a binary operation on $G $. Then the pair $(G, \oplus)$ is
called a {\it groupoid.}  A function $f$ from a groupoid $(G_1, \oplus_1)$ to a groupoid $(G_2, \oplus_2)$ is said to be
a groupoid homomorphism if $f(x_1\oplus_1 x_2)=f(x_1)\oplus_2 f(x_2)$ for any elements $x_1, x_2 \in G_1$.  In addition, a bijective
groupoid homomorphism from a groupoid $(G, \oplus)$ to itself will be called a groupoid automorphism. We will write $Aut (G, \oplus)$ for the set of all automorphisms of a groupoid $(G, \oplus)$.
\begin{definition}\cite{Ung}\label{Def:gyr}
 Let $(G, \oplus)$ be a nonempty groupoid. We say that $(G, \oplus)$ or just $G$
(when it is clear from the context) is a gyrogroup if the followings hold:
\begin{enumerate}
\item[($G1$)] There is an identity element $e \in G$ such that
$$e\oplus x=x=x\oplus e \text{~~~~~for all~~}x\in G.$$
\item[($G2$)] For each $x \in G $, there exists an {\it inverse element}  $\ominus x \in G$ such that
$$\ominus x\oplus x=e=x\oplus(\ominus x).$$
\item[($G3$)] For any $x, y \in G $, there exists an {\it gyroautomorphism} $\text{gyr}[ x, y ] \in Aut( G,  \oplus)$ such that
$$x\oplus (y\oplus z)=(x\oplus y)\oplus \text{gyr}[ x, y ](z)$$ for all $z \in G$;
\item[($G4$)] For any $x, y \in G$, $\text{gyr}[ x \oplus y, y ] = \text{gyr}[ x, y ]$.
\end{enumerate}
\end{definition}

In this paper, $\text{gyr}[a,b]V$ denotes $\{\text{gyr}[a,b]v: v\in V\}$.

The following Proposition \ref{Pro:gyr} below summarizes some algebraic properties of gyrogroups
\begin{proposition}(\cite{Ung1, Ung})\label{Pro:gyr}
Let $(G,\oplus)$ be a gyrogroup and $a,b,c\in G$. Then
\begin{enumerate}
\item[(1)] $\ominus(\ominus a)=a$; \hfill{Involution of inversion}
\item[(2)] $\ominus a\oplus(a\oplus b)=b$; \hfill{Left cancellation law}
\item[(3)] \text{gyr}$[a,b](c)=\ominus(a\oplus b)\oplus(a\oplus(b\oplus c))$; \hfill{Gyrator identity}
\item[(4)] $\ominus(a\oplus b)=\text{gyr}[a,b](\ominus b\ominus a)$;\hfill{\text{cf.~}$(ab)^{-1}=b^{-1}a^{-1}$}
\item[(5)] $(\ominus a\oplus b)\oplus \text{gyr}[\ominus a,b](\ominus b\oplus c)=\ominus a\oplus c$; \hfill{\text{cf.~}$(a^{-1}b)(b^{-1}c)=a^{-1}c$}
\item[(6)] $\text{gyr}[a,b]=\text{gyr}[\ominus b,\ominus a]$;\hfill{Even property}
\item[(7)] $\text{gyr}[a,b]=\text{gyr}^{-1}[b,a], \text{the inverse of gyr}[b,a]$. \hfill{Inversive symmetry}
\end{enumerate}
\end{proposition}

\begin{definition}\cite{Ung1}\label{De2.3}
Let $(G,\oplus)$ be a gyrogroup with gyrogroup operation (or,
addition) $\oplus$. The gyrogroup cooperation (or, coaddition) $\boxplus$ is a second
binary operation in $G$ given by the equation
$$(\divideontimes)~~~~a\boxplus b=a\oplus \text{gyr}[a,\ominus b]b$$ for all $a, b\in G$.
The groupoid $(G, \boxplus)$ is called a cogyrogroup, and is said to
be the cogyrogroup associated with the gyrogroup $(G, \oplus)$.

Replacing $b$ by $\ominus b$ in $(\divideontimes)$, along with Identity $(\divideontimes)$ we have the identity
$$a\boxminus b=a\ominus \text{gyr}[a,b]b$$ for all $a, b\in G$, where we use the obvious notation, $a\boxminus b = a\boxplus(\ominus b)$.
\end{definition}

\begin{theorem}\label{The1}\cite{Ung1}
Let $(G,\oplus)$ be a gyrogroup with cooperation $\boxplus$ given by Definition \ref{De2.3}.
 Then,
 \begin{enumerate}
\item[(1)] $a\oplus(\ominus a\oplus b)=b$; \hfill{Left cancellation law}
\item[(2)] $(b\ominus a)\boxplus a=b$; \hfill{(First) Right Cancellation Law}
\item[(3)] $(b\boxminus a)\oplus a=b$.\hfill{(Second) Right Cancellation Law}
\end{enumerate}
\end{theorem}

\begin{theorem}\label{The2}\cite{Ung1}
 Any gyrogroup $(G,\oplus)$ possesses the cogyroautomorphic inverse property,
$$\ominus(a\boxplus b)=(\ominus b)\boxplus(\ominus a)$$ for any $a, b\in G$.
\end{theorem}

Atiponrat \cite{Atip} extended the idea of topological groups to topological gyrogroups as following:
\begin{definition}\cite{Atip}
A triple $( G, \tau,  \oplus)$ is called a {\it topological gyrogroup} if and only if
\begin{enumerate}
\item[(1)] $( G, \tau)$ is a topological space;
\item[(2)] $( G, \oplus)$ is a gyrogroup;
\item[(3)] The binary operation  $\oplus: G  \times G  \rightarrow G$ is continuous where $G \times G$ is endowed with the product topology
and the operation of taking the inverse $\ominus(\cdot ) : G  \rightarrow G $, i.e. $x \rightarrow \ominus x$, is continuous.
\end{enumerate}
\end{definition}

If a triple $( G, \tau,  \oplus)$ satisfies the first two conditions and its binary operation is continuous, we call such
triple a {\it paratopological gyrogroup} \cite{Atip1}. Sometimes we will just say that $G$ is a topological gyrogroup (paratopological gyrogroup) if the binary operation and the topology are clear from the context. A topological gyrogroup $G$ is strong if there exists an open base $\mathcal {U}$ at the identity $e$
of $G$ such that gyr$[x, y](U) = U$ for all $x, y \in G$, $U \in \mathcal {U}$. In this case, we say that $G$ is a {\it strongly topological
gyrogroup} \cite{Bao} with an open base $\mathcal {U}$ at $e$. Clearly, every strongly topological gyrogroup is a topological gyrogroup.

\begin{proposition}\label{Pro1}\cite{Atip1}
Let $G$ be a paratopological gyrogroup and $A$ be an open set. Then $B\oplus A$ is open for each $B\subseteq G$.
\end{proposition}


In the following theorem, we characterize the families of subsets of a gyrogroup $G$ which can appear as neighborhood bases of the neutral element in topological gyrogroups, which likes the Pontrjagin conditions in topological groups. This result can be find in \cite{JX}. For the sake of completion, we give the detailed process of proof of Proposition \ref{the}.

\begin{lemma}\label{LEM2.8}\cite{JX}
Let $G$ be a topological gyrogroup and $x\in G$. Then $L_x^\boxplus(\cdot):G\rightarrow G$ is homeomorphisms, where $L_x^\boxplus(\cdot)$ is defined as: $L_x^\boxplus(y)=x\boxplus y$ for each $y\in G$.
\end{lemma}

\begin{proof}
According to the definition, we have that
\begin{align*}
x\boxplus y&=x \oplus \text{gyr}[x,\ominus y]y
\\&=x \oplus(\ominus(x\ominus y)\oplus(x\oplus(\ominus y \oplus y)))
\\&=x \oplus(\ominus(x\ominus y)\oplus x)
\end{align*}
Hence, $L_x^\boxplus(y)=L_x(R_x(\ominus(L_x(\ominus(y)))))$. Since the operations $L_x, R_x$ and $\ominus$ are homeomorphisms, so is their the compositions.
\end{proof}

\begin{proposition}\label{the}\cite{JX}
Let $G$ be a Hausdorff topological gyrogroup and $\mathcal{U}$ an open base at the neutral element $e$ of $G$.
Then the following conditions hold:
\begin{enumerate}
\item[(1)] for every $U\in\mathcal{U}$, there exists an element $V\in \mathcal{U}$ such that $V\oplus V\subseteq U$;
\item[(2)] for every $U\in\mathcal{U}$, and every $x\in U$, there exists $V\in \mathcal{U}$ such that $x\oplus V\subseteq U$;
\item[(3)] for every $U\in\mathcal{U}$ and $x\in G$, there exists $V\in \mathcal{U}$ such that $\ominus x\oplus(V\oplus x)\subseteq U$;
\item[(4)] for $U, V\in\mathcal{U}$, there exists $W\in\mathcal{U}$ such that $W\subseteq U\cap V$;
\item[(5)] for every $U\in\mathcal{U}$ and $a,b\in G$, there exists an element $V\in \mathcal{U}$ such
that gyr$[a,b]V\subseteq U$;
\item[(6)] for every $U\in\mathcal{U}$ and $b\in G$, there exists an element $V\in \mathcal{U}$ such that $\bigcup_{v\in V}\text{gyr}[v,b]V\subseteq U$;
\item[(7)] $\{e\}=\bigcap_{U\in\mathcal{U}} (U\boxminus U)$.
\item[(8)] for every $U\in\mathcal{U}$ and $x\in G$, there exists $V\in \mathcal{U}$ such that $V\boxplus x\subseteq x\oplus U$ and $x \oplus V \subseteq x\boxplus U$;
\item[(9)] for every $U\in\mathcal{U}$, there exists $V\in \mathcal{U}$ such that $\ominus V\subseteq U$.
\end{enumerate}

Conversely, let $G$ be a gyrogroup and let $\mathcal{U}$ be a family of subsets such that every element in which contains the neutral element $e$ in $G$ and satisfying conditions (1)-(9).
Then the family $\mathcal{B}_{\mathcal{U}}=\{a\oplus U:a\in G, U\in \mathcal{U}\}$ is a base for a Hausdorff topology $\mathcal{T}_{\mathcal{U}}$ on $G$.
With this topology, $G$ is a topological gyrogroup.
\end{proposition}
\begin{proof}
Let $U\in\mathcal{U}$.

(1) Since $G$ is a topological gyrogroup, the operation $op_2:G\times G\rightarrow G$ defined by $op_2(x,y)=x\oplus y$ is continuous.
Because $e\oplus e=e$, and $U$ is a neighborhood of $e$, there exist neighborhoods $O$ and $W$ of $e$ such that $O\oplus W\subseteq U$.
We choose $V\in \mathcal{U}$ such that $V\subseteq O\cap W$.
Then $V\oplus V\subseteq W$.

(2) Let $x\in U$.
We define $R_x:G\rightarrow G$ by $R_x(y)=x\oplus y$.
Since $R_x(e)=x$ and $R_x$ is continuous at $e$, there exists $V\in\mathcal{U}$ such that $x\oplus V=R_x(V)\subseteq U$.

(3) For every $x\in G$, we define left translation map $L_{\ominus x}:G\rightarrow G$ by $L_{\ominus x}(y)=\ominus x\oplus y$.
By the continuous of $L_{\ominus x}$, $L_{\ominus x}(x)=e$  and $U$ is a neighborhood of $e$, there exists a neighborhood $V'$ of $x$ such that $L_{\ominus x}(V')\subseteq U$,
that is $\ominus x\oplus V'=L_{\ominus x}(V')\subseteq U$.
We also define the right translation map $R_{x}:G\rightarrow G$ by $R_{x}(y)=y\oplus x$.
Then $R_{x}(e)=x$.
Because $R_{x}$ is continuous at $e$, for the neighborhood $V'$ of $x$, there exists $V\in \mathcal{U}$ such that $R_{x}(V)\subseteq V'$,
that is $V\oplus x=R_{x}(V)\subseteq V'.$
So we get $\ominus x\oplus (V\oplus x)\subseteq \ominus x\oplus V'\subseteq U$.

(4) It is clear since $\mathcal{U}$ is an open base at $e.$

(5) For every $a, b\in G$,
we define $f_{a,b}:G\rightarrow G$ by $f_{a,b}(x)=\text{gyr}[a,b] x$.
Since $f_{a,b}(e)=e$ and $f_{a,b}$ is continuous at $e$,
for every $U\in\mathcal{U}$, there exists $V\in\mathcal{U}$ such that $f_{a,b}(V)\subseteq U$, that is, gyr$[a,b] V\subseteq U$.

(6) Take $W\in\mathcal{U}$ such that $W\oplus W\subseteq U$.
Then $b\oplus(W\oplus W)$ is an open set containing $b$.
Since $G$ is a topological gyrogroup, one can find $V\in\mathcal{U}$ such that $$(b\oplus V)\oplus V\subseteq b\oplus(W\oplus W)\quad\quad\quad(*)$$
Note that $$(b\oplus V)\oplus V=b\oplus (V\oplus \bigcup_{v\in V}\text{gyr}[v,b]V)\quad\quad\quad(**)$$
By (*) and (**) we have
$b\oplus(V\oplus \bigcup_{v\in V}\text{gyr}[v,b]V)\subseteq b\oplus(W\oplus W)$, which means $V\oplus \bigcup_{v\in V}\text{gyr}[v,b]V\subseteq W\oplus W$.
So we can get $\bigcup_{v\in V}\text{gyr}[v,b]V\subseteq W\oplus W\subseteq U$.

(7) We assume that $G$ is Hausdorff.
If $\bigcap_{U\in\mathcal{U}} (U\boxminus U)\neq\{e\}$, then there is $x\in \bigcap_{U\in\mathcal{U}} (U\boxminus U)$ such that $x\neq e$.
Since $G$ is Hausdorff, there are an open set $V_1$ containing $x$ and a open set $V\in \mathcal{U}$ such that $V_1\cap V=\emptyset$. Since $V_1$ is a neighbourhood of $x$, $x\oplus e=x$ and $G$ is a topological gyrogroup,  one can find $U\in\mathcal{U}$ such that $(x\oplus U)\subseteq V_1$, hence we have that $(x\oplus U)\cap V=\emptyset$.
We choose $W\in \mathcal{U}$ such that $W\subseteq U\cap V$.
Then $(x\oplus W)\cap W=\emptyset$, that is, $x\notin W\boxminus W$, which is a contradiction.

For (8). Since $G$ is a topological gyrogroup,  it is obvious that $(x\oplus U)\ominus x$ is an open set containing the the neutral element $e$ of $G$.  Hence there is a $V_1\in \mathcal{U}$ such that $V_1\subseteq (x\oplus U)\ominus x$, which is equivalent to $V_1\boxplus x\subseteq x\oplus U$. By Lemma \ref{LEM2.8}, we have that $x\boxplus U$ is an open set containing $x$. Since the operation $L_x$ is continuous and $L_x(e)=x$, one can find $V_2\in \mathcal{U}$ such that $L_x(V_2)=x \oplus V_2 \subseteq x\boxplus U$. Take a $V\in \mathcal{U}$ such that $V\subseteq V_1\cap V_2$. Then the set $V$ is the required.

For (9). Since $G$ is a topological gyrogroup, the operation $\ominus$ is continuous. Clearly, $U$ is an open set containing the neutral element $e$, so one can find $V\in\mathcal{U}$ such that $\ominus V\subseteq U$.

To prove the converse, let $\mathcal{U}$ be a family of subsets of $G$ such that conditions (1)-(9) hold.
Let $\mathcal{T}=\{W\subseteq G: \text{for every~} x\in W~\text{there exists~} U\in\mathcal{U}~\text{such that~} x\oplus U\subseteq W\}.$

{\bf Claim 1.} $\mathcal{T}$ is a topology on $G$.
It is clear that $G\in\mathcal{T}$ and $\emptyset\in\mathcal{T}$.
It also easy to see that $\mathcal{T}$ is closed under unions.
To show that $\mathcal{T}$ is closed under finite intersections, let $V,W\in\mathcal{T}$.
Let $x\in V\cap W$. Since $x\in V\in\mathcal{T}$ and $x\in W\in\mathcal{T}$, there exist $O,Q\in\mathcal{U}$ such that $x\oplus O\subseteq V$ and $x\oplus Q\subseteq W$.
From (5) it follows that there exists $U\in\mathcal{T}$ such that $U\subseteq O\cap Q$.
Then, we have $x\oplus U\subseteq V\cap W$.
Hence, $V\cap W\in\mathcal{T}$, and $\mathcal{T}$ is a topology on $G$.

{\bf Claim 2.} If $O\in \mathcal{U}$ and $g\in G$, then $g\oplus O\in\mathcal{T}$.

Take any $x\in g\oplus O$, then $\ominus g\oplus x\in O$.
By property (2), there exists $V'\in\mathcal{U}$ such that $\ominus g\oplus x\oplus V'\subseteq O$.
For $V'$ and $\ominus g, x\in G$, there exists $V\in\mathcal{U}$ such that $\text{gyr}[\ominus g,x]V\subseteq V'$ by condition (5).
So we have $\ominus g\oplus(x\oplus V)=(\ominus g\oplus x)\oplus \text{gyr}[\ominus g,x]V\subseteq O$,
that is $x\oplus V\subseteq g\oplus O$.
Hence $g\oplus O\in\mathcal{T}$.

{\bf Claim 3.} The family $\mathcal{B}_{\mathcal{U}}=\{a\oplus U:a\in G, U\in \mathcal{U}\}$ is a base for the topology $\mathcal{T}$ on $G$.

Indeed, it follows from Claim 2 and the definition of $\mathcal{T}$.

{\bf Claim 4.} The multiplication in $G$ is continuous with respect to the topology $\mathcal{T}$.

Let $a, b$ be arbitrary elements of $G$, and $O$ be any element of $\mathcal{T}$ such that $a\oplus b\in O$.
Then there exists $W\in\mathcal{U}$ such that $(a\oplus b)\oplus W\subseteq O$.
There exists $U\in\mathcal{U}$ such that $a\oplus b\oplus \text{gyr}[a,b]U\subseteq(a\oplus b)\oplus W$ by condition (5).
For $U$ there exists $U_1\in\mathcal{U}$ such that $U_1\oplus U_1\subseteq U$.
For $b$ and $U_1$ there exists $U_2\in\mathcal{U}$ such that $\bigcup_{v\in U_2}\text{gyr}[v,b]U_2\subseteq U_1$ by condition (6).
By condition (4), we can get $U_3\subseteq U_1\cap U_2$.
For $U_3\in\mathcal{U}$, apply (3) to choose $U_4\in\mathcal{U}$ such that $U_4\oplus b\subseteq b\oplus U_3$.
Using condition (6) we can get $U_5\in\mathcal{U}$ such that $\bigcup_{v\in U_5}\text{gyr}[v,b]U_5\subseteq U_3$.
By the condition (4), we get $U_6\subseteq U_4\cap U_5$.
We have
\begin{align*}
\\& a\oplus U_6\oplus (b\oplus U_6)) \quad
\\&= a\oplus U_6\oplus \text{gyr}[a,e](b\oplus U_6)) \quad
\\&\subseteq a\oplus U_6\oplus \bigcup_{v\in U_6}\text{gyr}[a,v](b\oplus U_6)) \quad
\\&= a\oplus (U_6\oplus (b\oplus U_6)) \quad
\\&= a\oplus ((U_6\oplus b)\oplus \bigcup_{v\in U_6}\text{gyr}[v,b]U_6))) \quad
\\&\subseteq a\oplus ((U_4\oplus b)\oplus \bigcup_{v\in U_5}\text{gyr}[v,b]U_5))) \quad
\\&\subseteq a\oplus ((U_4\oplus b)\oplus U_3)  \quad \quad\quad\quad\quad\quad
\\&\subseteq a\oplus ((b\oplus U_3)\oplus U_3)  \quad \quad\quad\quad\quad\quad
\\&= a\oplus (b\oplus (U_3\oplus \bigcup_{v\in U_3}\text{gyr}[v,b]U_3))) \quad
\\&\subseteq a\oplus (b\oplus (U_1\oplus \bigcup_{v\in U_2}\text{gyr}[v,b]U_2))) \quad
\\&\subseteq a\oplus (b\oplus (U_1\oplus U_1)) \quad
\\&\subseteq a\oplus (b\oplus U)  \quad \quad\quad\quad\quad\quad
\\&=a\oplus b\oplus \text{gyr}[a,b]U\quad\quad\quad
\\&\subseteq(a\oplus b)\oplus W \quad\quad\quad
\end{align*}
Since $U_6\in\mathcal{U}$, $a\oplus U_6, b\oplus U_6$ are the neighborhood of $a, b$.
Thus, the multiplications in $G$ is continuous with respect to the topology $\mathcal{T}$.
This proves Claim 4.

{\bf Claim 5.} The gyrogroup $G$ with topology $\mathcal{T}$ is Hausdorff.

For every $x,y\in G$ and $x\neq y$, then $\ominus y\oplus x\neq e$.
There exist $U\in\mathcal{U}$ such that $\ominus y\oplus x\notin U\boxminus U$, which implies $\ominus y\oplus x\oplus U\cap U=\emptyset$.
For $x,y\in G$ and $U\in\mathcal{U}$, there exists $V\in\mathcal{U}$ such that $\text{gyr}[\ominus y, x]V\subseteq U$ by condition (5).
Then we claim that $x\oplus V\cap y\oplus U=\emptyset$, which implies that the gyrogroup $G$ with the topology $\mathcal{T}$ is Hausdorff.

In fact, if  $x\oplus V\cap y\oplus U\neq \emptyset$, then $\ominus y \oplus(x\oplus V)\cap U\neq \emptyset$. Hence, we have that $(\ominus y \oplus x)\oplus \text{gyr}[\ominus y, x]V\cap U \neq \emptyset$. Since $\text{gyr}[\ominus y, x]V\subseteq U$, we have that $(\ominus y \oplus x)\oplus U\cap U \neq \emptyset$. This is a contradiction.

{\bf Claim 6.} The inverse operation $\ominus: (G, \mathcal{T})\rightarrow (G, \mathcal{T})$ is continuous.

Take any $x\in G$ and any $U\in \mathcal{U}$. By the condition (8), there is $U_1\in \mathcal{U}$ such that $U_1\boxplus(\ominus x)\subseteq \ominus x\oplus U$. For $U_1$, applying the condition (9), one can find $U_2\in \mathcal{U} $ such that $\ominus U_2\subseteq U_1$. For $U_2$, applying the condition (8) again, one can find $V\in  \mathcal{U}$ such that $x\oplus V\subseteq x \boxplus U_2$. Then we have that
\begin{align*}
\ominus(x\oplus V)&\subseteq \ominus(x \boxplus U_2)
\\&=\ominus U_2 \boxplus (\ominus x)
\\&\subseteq U_1\boxplus (\ominus x)
\\&\subseteq \ominus x\oplus U
\end{align*}
Thus we have proved that the inverse operation $\ominus$ is continuous.
We prove that $G$ is a Hausdorff topological gyrogroup with the topology $\mathcal{T}$.
\end{proof}

\begin{remark}
Let $G$ be a Hausdorff paratopological gyrogroup and $\mathcal{U}$ an open base at the neutral element $e$ of $G$.
Then the conditions (1)-(7) in Proposition \ref{the} hold.
Conversely, let $G$ be a gyrogroup and let $\mathcal{U}$ be a family of subsets of $G$ satisfying conditions (1)-(7) of Proposition \ref{the}.
Then the family $\mathcal{B}_{\mathcal{U}}=\{a\oplus U:a\in G, U\in \mathcal{U}\}$ is a base for a Hausdorff topology $\mathcal{T}_{\mathcal{U}}$ on $G$.
With this topology, $G$ is a paratopological gyrogroup.
\end{remark}

\section{Embeddings into path-connected, locally path-connected gyrogroups}
Using the sufficient conditions of Proposition \ref{the}, we can topologize the extended gyrogroup $G^\bullet$ in the case when $G$ is a
topological gyrogroup, as shown in the following theorem. Let $G$ be a topological gyrogroup.
Given an open neighborhood $U$ of $e$ in $G$ and a real number $\varepsilon> 0$, define
$$O(U, \varepsilon) = \{f\in G^\bullet| \mu(\{r \in J|f(r) \notin U\}) < \varepsilon \},$$
where $\mu$ is the Lebesgue measure on the real line.

\begin{theorem}
Let $G$ be a Hausdorff topological gyrogroup and $\mathcal{U}$ an open base at the neutral element $e$ of $G$. Then, the family
$$\{f \oplus O(U, \varepsilon)| U\in \mathcal{U}, \varepsilon > 0 \text{~and~} f\in G^\bullet\}$$
forms a base of a topology on $G^\bullet$, and $G^\bullet$ becomes a Hausdorff topological gyrogroup.
\end{theorem}

\begin{proof}
It suffices to show that the family $\mathcal{U}^\bullet = \{O(U, \varepsilon) | U \in \mathcal{U}, \varepsilon> 0\}$ satisfies the sufficient conditions (1)-(9) of Proposition \ref{the}.


(1) Let $O(U, \varepsilon) \in \mathcal{U}^\bullet$. Then, there is a set $V \in \mathcal{U}$ such that $V\oplus V\subseteq U$ by Proposition 2.9(1).
Put $f,g\in O(V, \frac{\varepsilon}{2})$. Since $\{r \in J | f (r) \oplus g(r) \notin U\} \subseteq\{r\in J | f (r) \notin V\}\cup\{ r\in J | g (r) \notin V\}$,
we can get
$\mu(\{r \in J | f (r) \oplus g(r) \notin U\})\leq\mu(\{r\in J | f (r) \notin V\})+\mu(\{ r\in J | g (r) \notin V\})<\frac{\varepsilon}{2})+\frac{\varepsilon}{2})=\varepsilon.$
Thus we have $f \oplus g \in O(U, \varepsilon)$, which implies $O(V, \frac{\varepsilon}{2})\oplus O(V, \frac{\varepsilon}{2})\subseteq O(U, \varepsilon)$.

(2) Let $O(U, \varepsilon) \in \mathcal{U}^\bullet$ and $f \in O(U, \varepsilon)$.
Then, there is a partition $\{a_0, a_1, ..., a_n\}$ of $J$ such that $f$ is
constant on each interval $[a_k, a_{k+1})$. Set $L =\{k \in\{0, 1, ... , n-1\}| f (a_k)\in U\}.$
By Proposition \ref{the}(2),
for each $k\in L$, there is a set $V_k\in \mathcal{U}$ such that $f(a_k)\oplus V_k\subseteq U$.
By Proposition \ref{the}(4), there is a set $V\in \mathcal{U}$ such that $V\subseteq\bigcap_{k\in L}V_k$.
So, $f(r) \oplus V\subseteq U$ whenever $f(r)\in U$.
We take $\delta=\varepsilon-\mu(\{r\in J|f(r)\notin U\})$.
Let $ g\in O(V, \delta)$.
We have $\{r \in J | f (r) \oplus g(r) \notin U\}\subseteq\{r \in J | f (r) \notin U\}\cup\{r \in J | g(r) \notin V\}$.
It follows that $\mu(\{r \in J | f (r) \oplus g(r) \notin U\})\leq\mu(\{r \in J | f (r) \notin U\})+\mu(\{r \in J | g(r) \notin V\})<\mu(\{r \in J | f (r) \notin U\})+\delta=\varepsilon$.
Thus, $f \oplus g \in O(U, \varepsilon)$, which implies $f \oplus O(V, \delta)\subseteq O(U,\varepsilon)$.

(3) Let $O(U, \varepsilon) \in \mathcal{U}^\bullet$ and $f \in G^\bullet$.
Then, there is a partition $\{a_0, a_1, ..., a_n\}$ of $J$ such that $f$ is
constant on each interval $[a_k, a_{k+1})$.
By Proposition \ref{the}(3),
for each $k=0, 1, ... , n-1$, there is a set $V_k\in \mathcal{U}$ such that $\ominus f(a_k)\oplus(V_k\oplus f(a_k))\subseteq U$.
By Proposition 2.9(4), there is a set $V\in \mathcal{U}$ such that $V\subseteq\bigcap^{n-1}_{i=0}V_k$.
So, $\ominus f(r)\oplus(V\oplus f(r))\subseteq U$ for all $r\in J$.
Let $ g\in O(V, \varepsilon)$.
We have $\{r \in J |\ominus f(r)\oplus(g(r)\oplus f(r))\notin U\}=\{r \in J |g(r)\notin (f(r)\oplus U)\boxminus f(r)\}\subseteq\{r \in J |g(r) \notin V\}$.
Thus, $\ominus f\oplus(O(V, \varepsilon)\oplus f)\subseteq O(U,\varepsilon)$.

(4) Let $O(U, \varepsilon), O(V, \delta) \in \mathcal{U}^\bullet$. By Proposition \ref{the}(4), there is a set $W\in \mathcal{U}$ such that $W\subseteq U\cap V$.

Let $U, V \in \mathcal{U}$, $f\in G^\bullet$, and $\varepsilon,\delta>0$. Note that if $U \subseteq V$, then $\{f\in J | f (r) \notin V\} \subseteq\{r\in J |f (r) \notin U\}.$
Therefore, $O(U,\varepsilon)\subseteq O(V,\varepsilon)$.
On the other hand, if $\varepsilon<\delta$, then $\mu(\{r\in J |f (r) \notin U\})<\varepsilon$ implies $\mu(\{r\in J |f (r) \notin U\})<\delta$.
This shows that $O(U,\varepsilon)\subseteq O(U,\delta)$.
We take $\delta_0=\text{min} \{\varepsilon,\delta\}$.
Thus, it follows above that $O(W,\delta_0)\subseteq O(U,\varepsilon)\cap O(V,\delta)$.

(5) Let $O(U, \varepsilon) \in \mathcal{U}^\bullet$ and $f,g \in G^\bullet$. Then, there is a partition $\{a_0, a_1, ..., a_n\}$ of $J$ such that both $f$ and $g$ is
constant on each interval $[a_k, a_{k+1})$.
By Proposition \ref{the}(5),
for each $k=0, 1, ... , n-1$, there is a set $V_k\in \mathcal{U}$ such that $\text{gyr}[f(a_k),g(a_k)]V\subseteq U$.
By Proposition \ref{the}(4), there is a set $V\in \mathcal{U}$ such that $V\subseteq\bigcap^{n-1}_{i=0}V_k$.
Furthermore, $\text{gyr}[f(r),g(r)]V\subseteq U$ for all $r\in J$.

Let $h\in O(V, \varepsilon)$. Then, $\{r \in J| \text{gyr}[f(r),g(r)]h(r)\notin U\}=\{r \in J| h(r)\notin \text{gyr}[g(r),f(r)]U\}\subseteq\{r\in J |h (r) \notin V\}.$
It follows that $\mu(\{r \in J| \text{gyr}[f(r),g(r)]h(r)\notin U\})\leq\mu(\{r\in J |h (r) \notin V\})<\varepsilon$.
This shows that $\text{gyr}[f,g]h\in O(U, \varepsilon)$, that is $\text{gyr}[f,g]O(V, \varepsilon)\subseteq O(U, \varepsilon)$.

(6) Let $O(U, \varepsilon) \in \mathcal{U}^\bullet$ and $f \in G^\bullet$.
Then, there is a partition $\{a_0, a_1, ..., a_n\}$ of $J$ such that $f$ is
constant on each interval $[a_k, a_{k+1})$.
By Proposition \ref{the}(6),
for each $k=0, 1, ... , n-1$, there is a set $V_k\in \mathcal{U}$ such that $\bigcup_{v\in V_k}\text{gyr}[v,f(a_k)]V_k\subseteq U$.
By Proposition \ref{the}(2), there is a set $V\in \mathcal{U}$ such that $V\subseteq\bigcap^{n-1}_{i=0}V_k$.
So, $\bigcup_{v\in V}\text{gyr}[v,f(r)]V\subseteq U$ for all $r\in J$.
Let $h, h'\in O(V, \varepsilon/2)$.
We have $\{r \in J |(\text{gyr}[h',f]h)(r)\notin U\}=\{r \in J |\text{gyr}[h'(r),f(r)]h(r)\notin U\}\subseteq\{r \in J |h'(r) \notin V\}\cup \{r \in J |h(r) \notin V\}$.
It follows that $\mu(\{r \in J |\text{gyr}[h'(r),f(r)]h(r)\notin U\})\leq\mu(\{r \in J |h'(r) \notin V\})+\mu(\{r \in J |h(r) \notin V\})<\varepsilon$.
This shows that $\text{gyr}[h',f]h\in O(U, \varepsilon)$, that is $\bigcup_{h'\in O(V, \varepsilon/2)}\text{gyr}[h',f]O(V, \varepsilon/2)\subseteq O(U, \varepsilon)$.


(7) Let $e^\bullet\neq f\in G^\bullet$. Then, there exists a subinterval
$[a, b)\subseteq J$ such that $f$ is constant on $[a, b)$ and $f(a) \neq e$.
Since $G$ is Hausdorff, it is also $T_1$. There is $U\in \mathcal{U}$ such that $f(a)\notin U$.
Then, $[a, b)\subseteq \{r \in J | f (r) \notin U\}.$
It follows that $b-a\leq \mu(\{r \in J | f (r) \notin U\})$. Thus, $f\notin O(U,b-a)$.
On the other hand,
for $G$ is a topological gyrogroup, then $op_2:G\times G\rightarrow G$ defined by $op_2(x,y)=x\boxminus y$ is continuous.
Because $e\boxminus e=e$, and $U\in \mathcal{U}$, there exist neighborhood $O$ and $W$ of $e$ such that $O\boxminus W\subseteq U$.
We choose $V\in \mathcal{U}$ such that $V\subseteq O\cap W$.
Then $V\boxminus V\subseteq W$.
Put $f,g\in O(V, \frac{b-a}{2})$. Since $\{r \in J | f (r)\boxminus g(r) \notin U\} \subseteq\{r\in J | f (r) \notin V\}\cup\{ r\in J | g (r) \notin V\}$,
we can get
$\mu(\{r \in J | f (r)\boxminus g(r) \notin U\})\leq\mu(\{r\in J | f (r) \notin V\})+\mu(\{ r\in J | g (r) \notin V\})<\frac{b-a}{2}+\frac{b-a}{2}=b-a.$
Thus we have $f\boxminus g \in O(U, b-a)$, which implies $O(V, \frac{b-a}{2})\boxminus O(V, \frac{b-a}{2})\subseteq O(U, b-a)$.
So we get $f\notin O(V, \frac{b-a}{2})\boxminus O(V, \frac{b-a}{2})$, which is a contradiction.

(8) Let $O(U, \varepsilon) \in \mathcal{U}^\bullet$ and $f \in G^\bullet$.
Then, there is a partition $\{a_0, a_1, ..., a_n\}$ of $J$ such that $f$ is
constant on each interval $[a_k, a_{k+1})$.
By Proposition \ref{the}(8),
for each $k=0, 1, ... , n-1$, there is a set $V_k\in \mathcal{U}$ such that $V_k\boxplus f(a_k)\subseteq f(a_k)\oplus U$ and $f(a_k) \oplus V_k \subseteq f(a_k)\boxplus U$.
By Proposition \ref{the}(4), there is a set $V\in \mathcal{U}$ such that $V\subseteq\bigcap^{n-1}_{i=0}V_k$.
So, $V\boxplus f(r)\subseteq f(r)\oplus U$ and $f(r) \oplus V \subseteq f(r)\boxplus U$ for all $r\in J$.
They are equivalent to $V\subseteq (f(r)\oplus U)\ominus f(r)$ and $V\subseteq \ominus f(r)\oplus (f(r)\boxplus U)$.
For all $v\in V$, there is $u\in U$ such that $$f(r) \oplus v =f(r)\boxplus u$$
if and only if $$(f(r) \oplus v )\ominus u=f(r)$$
if and only if $$\ominus u =\ominus(f(r) \oplus v) \oplus f(r)$$
if and only if $$ u =\ominus(\ominus(f(r) \oplus v) \oplus f(r)).$$

Let $h\in O(V, \varepsilon)$.
We have $\{r \in J |(\ominus f\oplus (h\boxplus f))(r)\notin U\}=\{r \in J |\ominus f(r)\oplus (h(r)\boxplus f(r))\notin U\}=\{h(r)\notin (f(r)\oplus U)\ominus f(r)\}\subseteq\{r \in J |h(r) \notin V\}$.
It follows that $\mu(\{r \in J |(\ominus f\oplus (h\boxplus f))(r)\notin U\})\leq\mu(\{r \in J |h(r) \notin V\})<\varepsilon$.
This shows that $\ominus f\oplus (h\boxplus f)\in O(U, \varepsilon)$, that is $O(V, \varepsilon)\boxplus f\subseteq f\oplus O(U, \varepsilon)$.

For the above $h\in O(V, \varepsilon)$,
\begin{align*}
&\{r \in J |(\ominus(\ominus(f \oplus h) \oplus f)(r)\notin U\}
\\&=\{r \in J |\ominus(\ominus(f(r) \oplus h(r)) \oplus f(r))\notin U\}
\\&=\{r \in J |\ominus(f(r) \oplus h(r)) \oplus f(r)\notin \ominus U\}
\\&=\{r \in J |\ominus(f(r) \oplus h(r)) \notin \ominus U\boxplus(\ominus f(r))\}
\\&=\{r \in J |f(r) \oplus h(r) \notin  f(r)\boxplus U\}
\\&=\{r \in J |h(r) \notin  \ominus f(r) \oplus (f(r)\boxplus U)\}
\\&=\{r \in J |h(r) \notin  V\}
\end{align*}
It follows that $\mu(\{r \in J |(\ominus(\ominus(f \oplus h) \oplus f)(r)\notin U\})\leq\mu(\{r \in J |h(r) \notin V\})<\varepsilon$.
This shows that $\ominus(\ominus(f \oplus h) \oplus f\in O(U, \varepsilon)$, that is $f \oplus O(V, \varepsilon) \subseteq f\boxplus O(U, \varepsilon)$.

(9) Let $O(U, \varepsilon) \in \mathcal{U}^\bullet$.
By Proposition \ref{the}(9), there is a set $V\in \mathcal{U}$ such that $\ominus V\subseteq U$.
Let $h\in O(V, \varepsilon)$.
We have $\{r \in J |\ominus h(r)\notin U\}=\{r \in J |h(r)\notin \ominus U\}\subseteq\{r \in J | h(r)\notin V\}$.
It follows that $\mu(\{r \in J |\ominus h(r)\notin U\})\leq\mu(\{r \in J | h(r)\notin V\})<\varepsilon$.
This shows that $\ominus h\in O(U, \varepsilon)$, that is $\ominus O(V, \varepsilon)\subseteq O(U, \varepsilon)$.

With this topology $\mathcal{T}_{\mathcal{U}^\bullet}$, $G^\bullet$ is a topological gyrogroup, since $\mathcal{U}^\bullet $ satisfies the sufficient conditions in Proposition 2.9.
\end{proof}

\begin{theorem}
$G^\bullet$ is path-connected and locally path-connected for any topological gyrogroup $G$.
\end{theorem}
\begin{proof}
It is obvious that if $O(U, \varepsilon)$ is path-connected for all $O(U, \varepsilon) \in \mathcal{U}^\bullet$, then $G^\bullet$ is locally path-connected
since every topological gyrogroup is a homogeneous space.

Let $O(U, \varepsilon) \in \mathcal{U}^\bullet$ and $f \in G^\bullet$. Then, there is a partition $\{a_0, a_1, ..., a_n\}$ of $J$ such that $f$ is
constant on each interval $[a_k, a_{k+1})$.
Given $t\in[0,1]$ and $k\in\{0, 1, ... , n-1\}$, put $b_{k,t}=a_k+t(a_{k+1}-a_k)$.
For each $r\in J$, there exists $k\in\{0, 1, ... , n-1\}$ such that $r \in[a_k, a_{k+1})$.
We define $f_t : J \rightarrow G$ as follows:
$$f_t(r)=\left\{
\begin{array}{rcl}
f(r)      &      & {\text{~if~} a_k\leq r\leq b_{k,t} \text{~for some~} k}\\
e    &      & {\text{~otherwise~}}\\
\end{array} \right. \quad\quad \quad\quad$$
Specially, $f_0=e^\bullet$ and $f_1=f$. Clearly, $f_t \in G^\bullet$ for all $t\in[0, 1]$.
Furthermore, we have
$$\{r \in J | f_t(r) \notin U\}\subseteq \{r \in J | f (r) \notin U\},$$
and so $f_t\in O(U, \varepsilon)$ for all $t \in[0, 1]$.
Define a function $\varphi: [0, 1]\rightarrow O(U, \varepsilon)$ by $\varphi(t) = f_t$ for all $t\in [0, 1]$.
To show that $\varphi$ is continuous, let $t \in[0, 1]$ and $f_t\oplus O(V, \delta)$ be an open neighborhood of $\varphi(t) = f_t$ for each $V\in \mathcal{U}$ and $\delta>0$.
For each $s\in(t - \delta, t +\delta)\cap[0, 1]$.
We have $f_s\in f_t\oplus O(V, \delta)$, since
$\{r \in J | (\ominus f_t\oplus f_s)(r)\notin V\}\subseteq \{r \in J | f_t(r) \neq f_s(r)\}$ and
\begin{align*}
\\& \mu(\{r \in J | f_t(r) \neq f_s(r)\})\quad
\\&\leq\sum_{k=0}^{n-1}|b_{k,s}-b_{k,t}| \quad
\\&=\sum_{k=0}^{n-1}|(s-t)(a_{k+1}-a_{k})| \quad
\\&=\sum_{k=0}^{n-1}|s-t||(a_{k+1}-a_{k})|  \quad
\\&= |s-t| \quad
\\&<\delta. \quad\quad\quad
\end{align*}
This shows that $\varphi$ is continuous. Therefore,
$O(U, \varepsilon)$ is path-connected, and so $G^\bullet$
is locally path-connected.
In particular, for $O(U, 2) = G^\bullet$. It follows that $G^\bullet$ is
path-connected.
\end{proof}

\begin{lemma}\label{LEM4.11}
Let $G$ and $H$ be topological gyrogroups. If the homomorphism function $f: G \rightarrow H$ is continuous and open at $e$, then it is continuous and open at every point of $G$.
\end{lemma}
\begin{proof}
Let $g\in G$.
So $f(g)=(L_{f(g)}\circ f\circ L_{\ominus g})(g)=L_{f(g)}(f( L_{\ominus g}(g)))=L_{f(g)}(f(e))$.
Since $L_{f(g)}$ and $L_{\ominus g}$ are homeomorphism, we get $f$ is continuous and open at $g.$
\end{proof}

\begin{proposition}\cite{Wat}\label{Pro1}
Let $G$ be a
gyrogroup. The function $i_G : G \rightarrow G^\bullet$ defined by
$$i_G(x) = x^\bullet, \text{~for every~} x\in G, $$
is a gyrogroup monomorphism, where $x^\bullet : J \rightarrow G$ defined by $x^\bullet(r) = x$ for all $r\in J$. Clearly, $x^\bullet \in G^\bullet$.
Consequently, $i_G(G)$ forms a subgyrogroup of $G^\bullet$ that is isomorphic to $G$ as gyrogroups.
\end{proposition}

\begin{lemma}\label{LEM}\label{LEM3.12}
For a topological gyrogroup $G$, the function $i_G : G \rightarrow G^\bullet$ defined in Proposition \ref{Pro1} is
continuous and open at $e$.
\end{lemma}
\begin{proof}
Let $\varepsilon\in(0, 1)$.
Note that for all $x\in G$, $O(V, \varepsilon)\in \mathcal{N}(e^\bullet)$,
\begin{align*}
\\&y^\bullet\in O(V,\varepsilon)  \quad
\\& \Leftrightarrow \mu(\{r\in J|y^\bullet(r)\notin V\})<\varepsilon \quad
\\&\Leftrightarrow\mu(\{r\in J|y\notin V\})<\varepsilon  \quad
\\&\Leftrightarrow y\in V.\quad\quad
\end{align*}
Then, for each $O(V, \varepsilon)\in \mathcal{N}(e^\bullet)$, we have $i_G(V)\subseteq O(V, \varepsilon).$
Thus, $i_G$ is continuous at $e$.
To show that $i_G$ is open at $e$, let $V\in \mathcal{N}(e)$.
For any $g\in V$, $i_G(g)=g^\bullet\in i_G(G)$. Note that
$\{r\in J|g^\bullet(r)\notin V\}=\{r\in J|g\notin V\}$.
Therefore, $\mu(\{r\in J|g^\bullet(r)\notin V\})<\frac{1}{2}$. We have $i_G(V)\subseteq i_G(G)\cap O(V, \frac{1}{2})$
For any $f\in i_G(G)\cap O(V, \frac{1}{2})$, there exists $g\in G$ such that $f=g^\bullet$.
We assert that $g\in V$. If $g\notin V$, $\mu(\{r\in J|g^\bullet(r)\notin V\})=\mu(\{r\in J|g\notin V\})=1$, which is a contradiction.
This shows that $i_G(G)\cap O(V, \frac{1}{2})\subseteq i_G(V)$, and so $i_G(V)=i_G(G)\cap O(V, \frac{1}{2})$.
Thus, $i_G$ is open at $e$.
\end{proof}

\begin{theorem}\label{the2}
For any topological gyrogroup $G$, the function $i_G$ defined in Proposition \ref{Pro1} is a
topological embedding and
$i_G(G)$ forms a closed subgyrogroup of $G^\bullet$.
\end{theorem}
\begin{proof}
By Lemma \ref{LEM4.11} and Lemma \ref{LEM3.12}, it is easy to see that $i_G$ is a
topological embedding.
The remaining part is to show that $i_G(G)$ is a closed subset of $G^\bullet$.
Take any $f\in G^\bullet\setminus i_G(G)$. Then there are numbers $a_1, a_2, a_3, a_4$
satisfying  $0\leq a_1<a_2<a_3<a_4\leq 1$ such that $f$ is constant on
$[a_1, a_2)$ and $[a_3, a_4)$ with $f(a_1)=x_1\neq x_2=f(a_3)$.
Therefore, there is an open set $V\in \mathcal{N}(e)$ such that $x_1\oplus V\cap x_2\oplus V=\emptyset$.
Put $\varepsilon=\text{~min~}\{a_2-a_1, a_4-a_3\}$. We claim that $i_G(G)\cap(f\oplus O(V, \varepsilon))=\emptyset$.
Otherwise, if $x^\bullet\in f\oplus O(V,\varepsilon)$ for some $x\in G$.
Then, $\ominus f\oplus x^\bullet\in O(V,\varepsilon)$.
We have $\mu(\{r\in J|\ominus(f(r))\oplus x\notin V\})<\varepsilon$.
So there exists $r_1\in[a_1, a_2)$ and $r_2\in[a_3, a_4)$ such that $\ominus(f(r_1))\oplus x\in V$ and $\ominus(f(r_2))\oplus x\in V$
which means $x\in (f(r_1))\oplus V=x_1\oplus V$ and $x\in (f(r_2))\oplus V=x_2\oplus V$. Thus, $x_1\oplus V\cap x_2\oplus V\neq\emptyset$, a contradiction.
Hence, $i_G(G)\cap(f\oplus O(V, \varepsilon))=\emptyset$.
\end{proof}





\section{Other results on the extension of topological gyrogroups}
In this section, we prove several topological properties shared by $G$ and $G^\bullet$, where $G$ is a
topological gyrogroup.  Proposition \ref{the4.5} shows that any two open bases of $G$ at the gyrogroup identity
generate the same topology on $G^\bullet$.

\begin{proposition}\label{the4.5}
Let $G$ be a topological gyrogroup with open bases $\mathcal{N}_1(e)$ and $\mathcal{N}_2(e)$ at $e$. Then, the
two bases
$$\mathcal{B}_1=\{f \oplus O(V, \varepsilon)| V\in \mathcal{N}_1(e), \varepsilon > 0 \text{~and~} f\in G^\bullet\}$$ and
$$\mathcal{B}_1=\{f \oplus O(V, \varepsilon)| V\in \mathcal{N}_2(e), \varepsilon > 0 \text{~and~} f\in G^\bullet\}$$
generate the same topology on $G^\bullet$.
\end{proposition}
\begin{proof}
Let $V_1\in \mathcal{N}_1(e)$ and $\varepsilon>0$. There is a set $V_2\in \mathcal{N}_2(e)$ such that $V_2\subseteq V_1$.
Hence, $O(V_2, \varepsilon)\subseteq O(V_1, \varepsilon)$.
This shows that $\mathcal{T}_{\mathcal{B}_1} \subseteq \mathcal{T}_{\mathcal{B}_2}$. In the same way, we have $\mathcal{T}_{\mathcal{B}_2} \subseteq \mathcal{T}_{\mathcal{B}_1}$.
This proves $\mathcal{T}_{\mathcal{B}_2}=\mathcal{T}_{\mathcal{B}_1}$.
\end{proof}

We extend the conclusion of Theorem 3 in \cite{Wat} to the topological version.

\begin{theorem}\label{The4.7}
Let $G$ and $H$ be topological gyrogroups. If $\varphi:G\rightarrow H$ is a continuous homomorphism,
then the function $\varphi^\bullet:G^\bullet\rightarrow H^\bullet$, defined by $\varphi^\bullet(f) =\varphi\circ f$ for all $f\in G^\bullet$, is continuous. If $\varphi$ is open, then so is $\varphi^\bullet$.
\end{theorem}
\begin{proof}
Let $f\in G^\bullet$
and $\varphi^\bullet( f ) \oplus O_H(V, \varepsilon)$ be a basic open neighborhood of $\varphi^\bullet(f)$ in $H^\bullet$. Since $\varphi$ is continuous, there is
$U\in \mathcal{N}(e)$ such that $\varphi(U) \subseteq V$. For any $h\in f\oplus O_G(U, \varepsilon)$, we have $\ominus f\oplus h\in O_G(U, \varepsilon)$.
Then,
\begin{align*}
\\&\{r\in J|\varphi^\bullet(\ominus f\oplus h)(r)\notin V\}  \quad
\\& =\{r\in J|\varphi\circ(\ominus f\oplus h)(r)\notin V\}  \quad
\\&\subseteq \{r\in J|(\ominus f\oplus h)(r)\notin U\}   \quad
\end{align*}
It follows that $\mu(\{r\in J|\varphi^\bullet(\ominus f\oplus h)(r)\notin V\})<\varepsilon$.
Thus, $\varphi^\bullet(\ominus f\oplus h)\in O_H(V, \varepsilon)$, which means $\varphi^\bullet(h)\in \varphi^\bullet(f)\oplus O_H(V, \varepsilon)$.
This shows that $\varphi^\bullet$ is continuous.

To show that $\varphi^\bullet$ is open, let
$O_G(V, \varepsilon)$ be a basic open neighborhood at $e^\bullet$ and let $f\in O_G(V, \varepsilon)$. Note that
$$\{r\in J|\varphi^\bullet(f)(r)\notin \varphi(V)\}=\{r\in J|\varphi\circ f(r)\notin \varphi(V)\}\subseteq \{r\in J|f(r)\notin V\}.$$
It follows that $\mu(\{r\in J|\varphi^\bullet(f)(r)\notin \varphi(V)\})<\varepsilon$, which means $\varphi^\bullet(f)\in O_H(\varphi(V), \varepsilon).$
This shows that $\varphi^\bullet(O_G(V, \varepsilon))\subseteq O_H(\varphi(V), \varepsilon)$.
 Let $g\in O_H(\varphi(V), \varepsilon)$.
Then, there is a partition $\{a_0, a_1, ..., a_n\}$ of $J$
such that $g([a_k, a_{k+1}))=h_k$ for $k= 0, 1,..., n-1$.
We get $x_k\in G$, $k= 0, 1,..., n-1$, such that $\varphi(x_k)=h_k$ and if $h_k\in \varphi(V)$, $x_k\in V$.

Define a function $f:J\rightarrow G$ by
$f(r)=x_k$, $r\in [a_k,a_{k+1}).$
Clearly, $f\in G^\bullet$.
Let $r\in J$. Then, $r\in[a_k, a_{k+1})$ for some $k$ and $\varphi^\bullet(f)(r) =\varphi(( f (r)) = \varphi(x_k)= g(r)$.

Note that
$$\{r\in J|f(r)\notin V\}\subseteq\{r\in J|g(r)\notin \varphi(V)\}=\{r\in J|\varphi^\bullet(f)(r)\notin \varphi(V)\}.$$
Since $g\in O_H(\varphi(V), \varepsilon)$, we have $\mu(\{r\in J|f(r)\notin V\})\leq \mu(\{r\in J|\varphi^\bullet(f)(r)\notin \varphi(V)\})<\varepsilon$.
Thus, $f\in O_G(V, \varepsilon)$.
This shows that $O_H(\varphi(V), \varepsilon)\subseteq \varphi^\bullet(O_G(V, \varepsilon))$,  and so $O_H(\varphi(V), \varepsilon)=\varphi^\bullet(O_G(V, \varepsilon)).$
It follows from Proposition \ref{the4.5} that $\varphi^\bullet$ is open.
\end{proof}

The following theorem is a topological version of \cite[Theorem 4]{Wat}.
\begin{theorem}
Let $G$ be a topological gyrogroup with open base $\mathcal{N}(e)$ at $e$. If $d$ is a bounded pseudometric
(respectively, metric) on $G$, then $d$ admits an extension to a bounded pseudometric (respectively, metric) $d^\bullet$ on
$G^\bullet$ such that
\begin{enumerate}
\item[(i)] if $d$ is continuous, then so is $d^\bullet$;
\item[(ii)] if $d$ is a metric generating the topology of $G$, then $d^\bullet$ also generates the topology of $G^\bullet$.
\end{enumerate}
\end{theorem}
\begin{proof}
It is obvious that a pseudometric
(respectively, metric) $d$ on $G$ admits an extension to a pseudometric (respectively, metric) $d^\bullet$ on
$G^\bullet$ by of \cite[Theorem 4]{Wat}.

(i) Suppose that $d$ is continuous and bounded.
To prove the continuity of $d^\bullet$, we
firstly show that for all $f\in G^\bullet$, $\varepsilon> 0$, there are a set $V \in\mathcal{N}(e)$ and a number $\delta> 0$ such that
$f\oplus O(V,\delta)\subseteq B_{d^\bullet}(f,\varepsilon).$
Suppose that $f(J)=\{z_1, z_2,..., z_n\}$. For each $i\in\{1, 2, ..., n\}$, because $d$ is continuous and
$d(z_i, z_i) =0\in[0,\frac{\varepsilon}{2})$), there is an open neighborhood $V_i$ of $z_i$ with $d(V_i \times V_i)\subseteq[0,\frac{\varepsilon}{2})$.
In particular, $d(\{z_i\}\times V_i)\subseteq [0,\frac{\varepsilon}{2})$ for all $i\in\{1, 2, ..., n\}$.
Since $V_i$ is an open neighborhood of $z_i$, we have $\ominus z_i\oplus V_i$ is an open neighborhood of $e$, and so there is a set $V\in \mathcal{N}(e)$ such that
$V\subseteq\bigcap_{i=1}^{n}(\ominus z_i\oplus V_i).$
If $v\in V$ and $z_i\in\{z_1, z_2,..., z_n\}$, we have $z_i\oplus v\in z_i\oplus V \subseteq z_i\oplus(\ominus z_i\oplus V_i) = V_i$. It follows that
$d(z_i,z_i\oplus v)<\frac{\varepsilon}{2}$ for all $v\in V$ and $z_i\in\{z_1, z_2,..., z_n\}$.
We claim that $f\oplus O(V,\frac{\varepsilon}{2})\subseteq B_{d^\bullet}(f,\varepsilon)$, that is, if $g\in O(V,\frac{\varepsilon}{2})$, then $d^\bullet( f , f\oplus g) < \varepsilon.$
Let $g\in O(V,\frac{\varepsilon}{2})$. Then, there is a partition $\{a_0, a_1, ... , a_m\}$ of $J$ such that $f$ and $g$ are constant on each interval $[a_k, a_{k+1}).$
For each $k \in\{0, 1, ...,m -1\}$, let $x_k$ and
$y_k$ be the values of $f$ and $g$ on $[a_k, a_{k+1})$, respectively. Note that $\{x-1, x_2, ..., x_{m-1}\}=\{z_1, z_2, ..., z_n\}$.
Set $L =\{k \in\{0, 1,...,m-1\}| y_k\in V\}$ and $M = \{0, 1,...,m-1\}\setminus L$. Note that if $k \in L$, then $d(x_k, x_k\oplus y_k)<\frac{\varepsilon}{2}$
and that if $k\in M$, then $d(x_k, x_k\oplus y_k) < 1$. Furthermore, we have $\{r\in J|g(r)\notin V\}=\bigcup_{k\in M}[a_k, a_{k+1})$, and so $\sum_{k\in M}(a_{k+1}-a_k)<\frac{\varepsilon}{2}$.
By definition of $d^\bullet$,
\begin{align*}
\\&d^\bullet(f,f\oplus g)=\sum_{k=0}^{m-1}(a_{k+1}-a_k)d(x_k, x_k\oplus y_k) \quad
\\& =\sum_{k\in L}(a_{k+1}-a_k)d(x_k, x_k\oplus y_k)+\sum_{k\in M}(a_{k+1}-a_k)d(x_k, x_k\oplus y_k) \quad
\\&<\sum_{k\in L}(a_{k+1}-a_k)\frac{\varepsilon}{2}+\sum_{k\in M}(a_{k+1}-a_k) \quad
\\&<\frac{\varepsilon}{2}+\frac{\varepsilon}{2}=\varepsilon. \quad\quad
\end{align*}
Therefore, $f\oplus O(V,\frac{\varepsilon}{2})\subseteq B_{d^\bullet}(f,\varepsilon)$.
Let $( f, g)\in G^\bullet \times G^\bullet$ and let $\varepsilon> 0$. Then, there are basic open sets $f\oplus O(U_1,\delta_1)$ and $g\oplus O(U_2,\delta_2)$
such that $f\in f\oplus O(U_1,\delta_1)\subseteq B_{d^\bullet}(f,\frac{\varepsilon}{2})$ and $g\in g\oplus O(U_2,\delta_2)\subseteq B_{d^\bullet}(g,\frac{\varepsilon}{2})$
If $f'\in f\oplus O(U_1,\delta_1)$ and $g'\in g\oplus O(U_2,\delta_2)$, then
\begin{align*}
\\&d^\bullet(f',g')\leq d^\bullet(f',f)+d^\bullet(f,g')\quad
\\& \leq d^\bullet(f',f)+d^\bullet(f,g)+d^\bullet(g,g') \quad
\\&<d^\bullet(f,g)+\varepsilon. \quad\quad
\end{align*}
On the other hand,
\begin{align*}
\\&d^\bullet(f,g)\leq d^\bullet(f,f')+d^\bullet(f',g)\quad
\\& \leq d^\bullet(f,f')+d^\bullet(f',g')+d^\bullet(g',g) \quad
\\&<d^\bullet(f',g')+\varepsilon. \quad\quad
\end{align*}
This shows that $d^\bullet(f,g)-\varepsilon < d^\bullet(f',g')<d^\bullet(f,g)+\varepsilon$, and so $d^\bullet(f',g')$ is in $(d^\bullet(f,g)-\varepsilon, d^\bullet(f,g)+\varepsilon)$.
Hence, $d^\bullet$ is continuous at $(f , g)$.

(ii) Finally, suppose that $d$ is a metric on $G$ generating the topology of
$G$. Let $f\in G^\bullet$ and let $f\oplus O(V, \varepsilon)$ be a basic open neighborhood of $f$ in $G^\bullet$.
Suppose that $f(J)=\{u_1, u_2,..., u_n\}$.
Then, there exists a number $\delta > 0$ such that $B_d(u_k, \delta)\subseteq u_k \oplus V$ for all $k = 1, 2,... n$. Note that if $1\leq k\leq n$
and $y\in G \setminus (u_k \oplus V)$, then $d(u_k, y) \geq \delta$. Put $\delta_0= \varepsilon\delta$ We claim that $B_{d^\bullet}(f, \delta_0)\subseteq f\oplus O(V, \varepsilon)$.
Let $g\in B_{d^\bullet}(f, \delta_0)$. Then, there exists a partition $\{b_0, b_1, ..., b_N\}$ of $J$ such that $f$ and $g$ are constant on
each interval $[b_i, b_{i+1})$. For each $i \in\{0, 1, ..., N-1\}$, let $x_i$ and $y_i$ be the values of $f$ and $g$ on $[b_i, b_{i+1})$,
respectively. Note that $\{u_1, u_2,..., u_n\}= \{x_0, x_1,..., x_{N-1}\}$. Set $P = \{i\in \{0, 1,... , N-1|y_i \notin x_i\oplus V\}.$
If $i\in P$, then $y_i\in G\setminus(x_i\oplus V)$, and so $d(x_i, y_i)\geq\delta$. It follows that
\begin{align*}
\\&\delta\sum_{i\in P}(b_{i+1}-b_i)=\sum_{i\in P}(b_{i+1}-b_i)\delta \quad
\\& \leq\sum_{i\in P}(b_{i+1}-b_i)d(x_i,y_i) \quad
\\&\leq\sum_{0\leq i<N-1}(b_{i+1}-b_i)d(x_i,y_i)  \quad
\\&=d^\bullet(f,g)  \quad
\\&<\delta_0. \quad\quad
\end{align*}
Hence, $\sum_{i\in P}(b_{i+1}-b_i)<\frac{\delta_0}{\delta}=\varepsilon$. Recall that $P = \{i\in\{0, 1, ..., N-1\}| g(b_i) \notin f (b_i)\oplus V\}$. We have
$\{r\in J|\ominus(f(r))\oplus g(r)\notin V\} =\bigcup_{i\in P}[b_i, b_{i+1}).$
Hence, $\mu(\{r\in J |\ominus(f (r))\oplus g(r)\notin V\}= \mu(\bigcup_{i\in P}[b_i, b_{i+1})) <\varepsilon$. This shows that $\ominus f \oplus g \in O(V, \varepsilon)$,
and so $g\in f\oplus O(V, \varepsilon)$. Hence, $B_{d^\bullet}(f,\delta_0)\subseteq f\oplus O(V, \varepsilon)$. Let $\mathcal{T}_{d^\bullet}$ be the topology on $G^\bullet$ induced by $d^\bullet$ and let $\mathcal{B} = \{B_{d^\bullet}(f,\varepsilon)| f \in G^\bullet, \varepsilon > 0\}$, which is a base for $\mathcal{T}_{d^\bullet}$. Hence, each basic open set $f\oplus O(V, \varepsilon)$ is a union of elements in $\mathcal{B}$. Let $\mathcal{T}$ be the topology on $G^\bullet$. Since $d^\bullet$ is continuous with respect to $\mathcal{T}$,
for each $f\in G^\bullet$, the function $F: G^\bullet\rightarrow[0,\infty)$ defined by $F(g)= d^\bullet(f , g)$ for all $g\in G^\bullet$ is continuous
(being the restriction of $d^\bullet$ to $\{f\}\times G^\bullet)$). It follows that $B_{d^\bullet}(f,\varepsilon)= F^{-1}([0, \varepsilon))$ is open with respect to
$\mathcal{T}$. Hence, $\mathcal{T}=\mathcal{T}_{d^\bullet}$.
\end{proof}

The following theorem proves that any continuous real-valued bounded function
on a topological gyrogroup $G$ can be extended to a continuous real-valued bounded function
on $G^\bullet$.

\begin{theorem}
For the function $F: G\rightarrow R$, $G$ is a topological gyrogroup, if $F$ is continuous and bounded,
then, $F$ admits an extension to a continuous real-valued bounded function on $G^\bullet$.
\end{theorem}
\begin{proof}
Let $G$ be a topological gyrogroup with open base $\mathcal{N}(e)$ at $e$ and let $F : G \rightarrow R$ be a
continuous and bounded function.
For each
$g\in G^\bullet$, there is a partition $\{ a_0, a_1, ...,a_n\}$ of $J$ such that $g$ is constant on each interval $[a_k, a_{k+1})$.
Define a function $F^\bullet: G^\bullet\rightarrow R$ by
$$F^\bullet(g)=\sum_{k=0}^{n-1}(a_{k+1}-a_k)F(g(a_k)).$$
It is easy to see $F$ is well defined and bounded.
Let $g\in G^\bullet$ and $\varepsilon>0$.
Since $F$ is continuous, there is a set $V\in \mathcal{N}(e)$ such that
$F(g(r)\oplus V)\subseteq (F(g(r))-\frac{\varepsilon}{2}, F(g(r))+\frac{\varepsilon}{2})$ for $r\in J.$

Let $f\in O(V, \frac{\varepsilon}{4})$. There is a partition $\{b_0, b_1, ...,b_n\}$ of $J$ such that $g, f$ are constant on each interval $[b_k, b_{k+1})$.
Set $L=\{k\in\{0, 1, ..., m -1\}| f(b_k)\notin V\}$. It follows that
$\{r\in J|f(r)\notin V\}=\cup_{k\in L}[b_k, b_{k+1})$. Since $f\in O(V,\frac{\varepsilon}{4})$,
$\sum_{k\in L}(b_{k+1}-b_k) = \mu(\{r \in J |f (r)\notin V\})<\frac{\varepsilon}{4}$ .
Now, consider
\begin{align*}
\\&|F^\bullet(g)-F^\bullet(g\oplus f)|=|\sum_{k=0}^{m-1}(b_{k+1}-b_k)(F(g(b_k))-F((g\oplus f)(b_k)))|  \quad
\\& \leq|\sum_{k\in L}(b_{k+1}-b_k)(F(g(b_k))-F((g\oplus f)(b_k)))|+|\sum_{k\notin L}(b_{k+1}-b_k)(F(g(b_k))-F((g\oplus f)(b_k)))| \quad
\\&\leq\sum_{k\in L}(b_{k+1}-b_k)|(F(g(b_k))-F((g\oplus f)(b_k)))|+\sum_{k\notin L}(b_{k+1}-b_k)|(F(g(b_k))-F((g\oplus f)(b_k)))| \quad
\\&<\frac{\varepsilon}{2}+\frac{\varepsilon}{2}=\varepsilon. \quad\quad
\end{align*}
Thus, $F^\bullet(g\oplus O(V, \frac{\varepsilon}{4}))\subseteq (F^\bullet(g)-\varepsilon, F^\bullet(g)+\varepsilon)$.
It follows $F^\bullet$ is continuous at $g$.
\end{proof}

\begin{theorem}
Let $G$ be a topological gyrogroup with open base $\mathcal{U}$ at $e$. $G$ is first countable if and only if $G^\bullet$ is first countable.
\end{theorem}
\begin{proof}
Suppose that $G$ is first countable. Then, there is a countable open base $\mathcal{U}'\subseteq\mathcal{U}$.
It is obvious that
the family $\{O(V, \frac{1}{n})|V\in\mathcal{U}',n\in\mathbb{N}\}$ is a countable open base at $e^\bullet$. Therefore, $G^\bullet$ is first
countable.

Conversely, let $G^\bullet$ is first countable.
By Proposition \ref{Pro1} and Theorem \ref{the2}, $G$ is topological embed into $G^\bullet$, so it is obvious that $G$ is first countable.
\end{proof}


\begin{thebibliography}{99}
\bibitem{Arha} A.V. Arhangel'ski, M. Tkachenko, Topological Groups and Related Structures, Atlantis Press and
World Sci., 2008.
\bibitem{Atip} W. Atiponrat, Topological gyrogroups: generalization of topological groups, Topol. Appl. 224 (2017) 73-82.
\bibitem{Atip1} W. Atiponrat, R. Maungchang, Complete regularity of paratopological gyrogroups, Topol. Appl. 270 (2020) 106951.
\bibitem{Ban} T. Banakh, A. Ravsky, Each regular paratopological group is completely regular, Proc. Am. Math. Soc. 145 (3) (2017)
1373-1382.
\bibitem{Bao} M. Bao, F. Lin, Feathered gyrogroups and gyrogroups with countable pseudocharacter. Filomat, 33(2019), 5113-5124.
\bibitem{Cai} Z. Cai, S. Lin, W. He, A note on Paratopological Loops, Bulletin of the Malaysian Math. Sci. Soc.,
42(5)(2019) 2535-2547.
\bibitem{Eng} R. Engelking, General Topology(revised and completed edition), Heldermann Verlag, Berlin, 1989.
\bibitem{Fer1}  M. Ferreira, Harmonic analysis on the Einstein gyrogroup, J. Geom. Symmetry Phys. 35 (2014) 21-60.
\bibitem{Fer2}  M. Ferreira, Harmonic analysis on the M$\ddot{o}$bius gyrogroup, J. Fourier Anal. Appl. 21 (2) (Apr 2015) 281-317.
\bibitem{JX} Y.Y. Jin, L.H. Xie, On paratopological gyrogroups, to appear.
\bibitem{HM} S. Hartman, J. Mycielski, On embedding of topological groups into connected topological groups, Colloq. Math. 5(1958) 167-169.
\bibitem{Ston} A.H. Stone, Paracompactness and product spaces, Bull. Amer. Math. Soc., 54(1948) 977-982.
\bibitem{Suk1} T. Suksumran, The Algebra of Gyrogroups: Cayley's Theorem, Lagrange's Theorem, and Isomorphism Theorems, Springer
International Publishing, Cham, (2016) 369-437.
\bibitem{Suk2} T. Suksumran, Gyrogroup actions: a generalization of group actions, J. Algebra 454 (2016) 70-91.
\bibitem{Suk3} T. Suksumran, K. Wiboonton, Isomorphism theorems for gyrogroups and $L$-subgyrogroups, J. Geom. Symmetry Phys. 37
(2015) 67-83.
\bibitem{Ung1} A.A. Ungar, Analytic hyperbolic geometry: Mathematical foundations and applications, World
Scientific, Hackensack, 2005.
\bibitem{Ung} A.A. Ungar, Analytic Hyperbolic Geometry and Albert Einstein's Special Theory of Relativity, World Scientific, 2008.
\bibitem{Wat} J. Wattanapan, W. Atiponrat, T. Suksumran, Embedding of Strongly Topological Gyrogroups in Path-Connected and Locally Path-Connected Gyrogroups, Symmetry. 12 (11) (2020): 1817-1840.
\end{thebibliography}
\end{document}